\documentclass[a4paper]{amsart}
\usepackage{lmodern}
\usepackage{amsthm}
\usepackage{amssymb}
\usepackage{amsmath}
\usepackage{hyperref}
\hypersetup{
  unicode=true,
  colorlinks=true,
  citecolor=blue,
  linkcolor=blue,
  anchorcolor=blue
}
\usepackage{todonotes}
\usepackage{tikz-cd}
\usepackage{dsfont,mathrsfs}
\usepackage{graphicx}

\usepackage{everyhook,xcolor}

\makeatletter
\DeclareRobustCommand\bigop[2][1]{%
  \mathop{\vphantom{\sum}\mathpalette\bigop@{{#2}{#1}}}\slimits@
}
\newcommand{\bigop@}[2]{\bigop@@#1#2}
\newcommand{\bigop@@}[3]{%
  \vcenter{%
    \sbox\z@{$#1\sum$}%
    \hbox{\resizebox{#3\dimexpr\ifx#1\displaystyle.9\fi\dimexpr\ht\z@+\dp\z@}{!}{$\m@th#2$}}%
  }%
}
\makeatother

\newcommand{\bigast}{\DOTSB\bigop[1.1]{*}}

\usepackage[noadjust,nocompress]{cite}

\newtheorem{theorem}{Theorem}[section]

\ifdefined\newcounteralias
  \newcommand{\newaliascnt}[2]{\relax}
  \newcommand{\aliascntresetthe}[1]{\relax}
  \let\oldtheorem\newtheorem
  \DeclareDocumentCommand{\newtheorem}{s m o m}{
    \IfBooleanTF{#1}%
    {\oldtheorem*{#2}{#4}}
    {\oldtheorem{#2}[theorem]{#4}}%
  }
\else
  \usepackage{aliascnt}
\fi

\newaliascnt{claim}{theorem}

\aliascntresetthe{claim}

\newaliascnt{lemma}{theorem}
\newtheorem{lemma}[lemma]{Lemma}
\aliascntresetthe{lemma}

\newaliascnt{proposition}{theorem}
\newtheorem{proposition}[proposition]{Proposition}
\aliascntresetthe{proposition}

\newaliascnt{corollary}{theorem}
\newtheorem{corollary}[corollary]{Corollary}
\aliascntresetthe{corollary}

\newaliascnt{fact}{theorem}
\newtheorem{fact}[fact]{Fact}
\aliascntresetthe{fact}

\theoremstyle{definition}
\newaliascnt{remark}{theorem}

\aliascntresetthe{remark}

\newaliascnt{definition}{theorem}
\newtheorem{definition}[definition]{Definition}
\aliascntresetthe{definition}

\newaliascnt{question}{theorem}
\newtheorem{question}[question]{Question}
\aliascntresetthe{question}

\newaliascnt{conjecture}{theorem}

\aliascntresetthe{conjecture}

\newtheorem*{theorem32}{\autoref{thm:cut}}
\newtheorem*{theorem38}{\autoref{thm:dc-it}}
\newtheorem*{theorem46}{\autoref{thm:knaster}}

\DeclareMathOperator{\dom}{dom}

\DeclareMathOperator{\id}{id}

\DeclareMathOperator{\tcl}{tcl}

\DeclareMathOperator{\fix}{fix}
\DeclareMathOperator{\aut}{Aut}
\DeclareMathOperator{\Add}{Add}
\DeclareMathOperator{\supp}{supp}

\let\leq\leqslant
\let\nleq\nleqslant

\newcommand{\ZF}{\mathsf{ZF}}
\newcommand{\ZFC}{\mathsf{ZFC}}

\newcommand{\DC}{\mathsf{DC}}

\newcommand{\AC}{\mathsf{AC}}

\newcommand{\power}{\mathcal P}

\newcommand{\cI}{\mathcal I}
\newcommand{\cL}{\mathcal L}

\newcommand{\sG}{\mathscr G}

\newcommand{\PP}{\mathbb P}
\newcommand{\QQ}{\mathbb Q}

\newcommand{\K}{\mathsf{K}}

\renewcommand{\neq}{\not=}

\newcommand{\forces}{\mathrel{\Vdash}}
\newcommand{\nforces}{\mathrel{\not\Vdash}}
\newcommand{\incompt}{\mathrel{\bot}}
\newcommand{\compt}{\mathrel{\|}}

\newcommand{\1}{\mathds 1}

\newcommand{\ccc}{\mathsf{CCC}}

\newcommand{\tup}[1]{\langle#1\rangle}

\date{November 18, 2026}
\title{Comments on Choiceless Chain Conditions}
\author{Constance Bromham}
\address[C.~Bromham]{School of Mathematics\\
  University of Leeds\\
  Leeds, LS2~9JT, UK}
\email{mm20cab@leeds.ac.uk}
\author{Asaf Karagila}
\address[A.~Karagila]{School of Mathematics\\
  University of Leeds\\
  Leeds, LS2~9JT, UK}
\email{karagila@math.huji.ac.il}
\urladdr{https://karagila.org}
\thanks{The second author was supported by UKRI Future Leaders Fellowship [MR/Y034058/1]}
\thanks{No data are associated with this article. For the purpose of open access, the authors have applied a CC-BY licence to any Author Accepted Manuscript version arising from this submission.}

\begin{document}
\begin{abstract}
  In \cite{ccc} it was suggested that the countable chain condition should be defined in the absence of choice as ``every predense set contains a countable predense subset''. During his tutorial at the ``120 Years of Choice'' conference in 2024, Philipp Schlicht remarked that this definition does not have an iteration theorem in $\ZF$. We provide a proof of this claim. Specifically, we show that if every finite support iteration of ccc forcings is again ccc, then the Axiom of Choice for countable families of countable sets must hold. However, we prove that under the assumption of the Principle of Dependent Choice, the finite support iteration of ccc forcings is again ccc. This requires correcting some issues with the definition of properness (and therefore Mekler's definition of ccc) from \cite{properness}. We then study Knaster principles and their relations to this version of ccc.
\end{abstract}
\maketitle
\section{Introduction}
In this paper, motivated by trying to understand Martin's Axiom without the Axiom of Choice, we study the theory of forcing without the Axiom of Choice. The technique of forcing is one of the most important techniques in the study of set theory. Developed by Paul J.\ Cohen \cite{Cohen1,Cohen2}, in order to prove the independence of Cantor's Continuum Hypothesis from the axioms of Zermelo--Fraenkel ($\ZF$), this technique allows us to start with a model of $\ZF$ and extend it in a ``controlled manner''. Moreover, this extension preserves ``nice'' and ``desirable'' properties of the model. For example, if the model is a transitive set of height $\alpha$, then so is the extension.

Cohen's strategy was to start with $M$, a countable transitive model of Zermelo--Fraenkel with the Axiom of Choice ($\ZFC$), then look at the partial order in $M$ which adds many new real numbers by finite approximations. Leaving aside the mechanics of why $\ZFC$ holds in the extension, if we added $\aleph_2$ many real numbers, we expect that the Continuum Hypothesis will fail in the extension. The only worry, of course, is that somehow by adding these new reals we have made changes to the cardinal structure of the extension so that the Continuum Hypothesis still holds, and the attempt was futile.

To avoid this problem Cohen proves that his partial order---known now as Cohen forcing---satisfies the \emph{countable chain condition}: any collection of conditions\footnote{Elements of these partial orders are called \emph{conditions}.} where any distinct two are incompatible must be countable. Using this he argues that there is no way to have collapsed the ordinal that was $\aleph_2$ to be an ordinal which is not a cardinal, as $M$ would be able to detect a set which is cofinal in $\aleph_2$ whose cardinality in $M$ is small, where none can exist when the Axiom of Choice is assumed.

This countable chain condition proved to be incredibly useful. Robert Solovay and Stanley Tennenbaum \cite{MA} developed the notion of iterated forcing in order to prove that Martin's Axiom is consistent. Iterating allows us to extend a model by forcing once, then extend the extension, and so on. This framework can be captured internally to the initial model using a partial order defined recursively. Moreover, when the iteration is done with a finite support, then if each partial order satisfied the countable chain condition, the final partial order will also satisfy the countable chain condition.

The concept of an iteration, and more precisely, a controlled and ``well-behaved'' iteration, became very important throughout set theory. Many classes of forcings will have iteration theorems which state ``the class $\Gamma$ is closed under iterations of type $I$''. Most famously, perhaps, is the theorem that a countable support iteration of proper forcing is proper.

While the Axiom of Choice is not used at all to define the mechanics of forcing,\footnote{Although it is preserved by forcing: extending a model of $\ZFC$ yields a model of $\ZFC$ again.} it is used heavily in many proofs that certain classes of partial orders are well-behaved. In recent years several papers studying the theory of forcing when the Axiom of Choice fails have been published (e.g., \cite{sequential,ccc,properness,morethings,IS}) two of which notably tried to understand the use of the Axiom of Choice in the application of chain conditions: Karagila and Schweber \cite{ccc} and Ikegami and Schlicht \cite{IS}. This work was motivated by trying to study Martin's Axiom in $\ZF$ (with, perhaps, a fragment of the Axiom of Choice), as an iteration theorem would be a first attempt in constructing a model where Martin's Axiom holds.

Our three main results deal with the definition of ccc given by Karagila and Schweber, called $\ccc_3$, which states that every predense set contains a countable predense subset. The first one formalises a comment made by Schlicht in his ``120 Years of Choice'' lectures in 2024.

\begin{theorem32}
  Suppose that the class of $\ccc_3$ forcings is closed under finite support iterations of length $\omega$. Then the countable union of countable sets is countable.
\end{theorem32}

As $\ZF$ does not prove that the countable union of countable sets is countable, this shows that the definition is indeed lacking a certain niceness when considered in $\ZF$. However, as the next theorem shows, this is remedied by adding the Principle of Dependent Choice ($\DC$).

\begin{theorem38}[$\DC$]
  Finite support iterations of $\ccc_3$ forcings are $\ccc_3$.
\end{theorem38}

In the presence of $\DC$, $\ccc_3$ is equivalent to Mekler's definition of ccc, which itself is derived from the concept of proper forcing, stating that every condition is $M$-generic for any suitable countable model. We identify and correct a mistake from \cite{properness} relating to the structure $H(\theta)$ that is often used for defining properness. Instead, we suggest a different class of structures, based on $V_\alpha$ instead. This allows us to prove the theorem using Mekler's definition.

Our final theorem deals with the Knaster property, which itself is a common strengthening of ccc. As Martin's Axiom proves that every ccc forcing is Knaster, this seemed like an excellent starting point for this project. It turns out that taking the Axiom of Choice out of the equation, much like in the case of ccc, introduces a new level of interesting constructions. We identify a few variants of the Knaster property in \autoref{sect:K}, denoted by $\K(i,j)$. The following theorem shows that in a sense the Knaster property is not actually stronger than ccc in $\ZF$.

\begin{theorem46}
  It is consistent with $\ZF+\DC$ that there is a notion of forcing which is $\K(2,1)$ but not $\ccc_3$.
\end{theorem46}

We finish the paper with several questions of interest about Martin's Axiom, iterations of ccc, and properness in $\ZF$.
\section{Preliminaries}
Our treatment of forcing is standard. We say that $\PP$ is a \emph{notion of forcing} if it is a preordered set with a maximum element, $\1_\PP$. The elements of $\PP$ are called \emph{conditions}, and given $p,q\in\PP$, we say that $q$ \emph{extends} $p$, or that it is a \emph{stronger} condition, if $q\leq p$.\footnote{We follow Goldstern's \emph{alphabet convention}, which states that a stronger condition will not have a letter appearing earlier in the alphabet than a weaker one.} Two conditions are compatible if they have a joint extension, and otherwise they are incompatible, in which case we write $p\incompt q$.

Given a family of forcing notions, $\{\QQ_i\mid i\in I\}$, we define the \emph{lottery sum} as the forcing $\bigoplus_{i\in I}\QQ_i$ whose conditions are $\bigcup_{i\in I}\{i\}\times\QQ_i$ and a maximum element $\1$, and $\tup{j,q_j}\leq\tup{i,q_i}$ if and only if $i=j$ and $q_j\leq_{\QQ_i}q_i$. That is, we allow the generic filter to ``choose'' which of the forcings we use, and then force with it.

Forcing can be iterated in $\ZF$ with a careful treatment. The preordered set $\PP\ast\dot\QQ$ is made of pairs, $\tup{p,\dot q}$ such that $p\in\PP$ and $\1_\PP\forces\dot q\in\dot\QQ$. The latter may seem like an appeal to the mixing lemma is needed, as often we think of such $\dot q$ as being mixed over various different names for conditions which appear in $\QQ$ itself. This would be problematic as the mixing lemma is equivalent to the Axiom of Choice, however, this can be easily fixed. Given any $\dot q$ such that $p\forces\dot q\in\dot\QQ$, we define $\dot q_*$ to be a name such that $p\forces\dot q=\dot q_*$, and whenever $p'\incompt p$, then $p'\forces\dot q_*=\dot\1_\QQ$. This can be done uniformly, and so we can uniformly produce---with $\dot\QQ$ as a parameter---a canonical set of names, $Q$, for conditions in $\QQ$, such that $\1_\PP\forces\dot q\in\dot\QQ$ for all $\dot q\in Q$, and whenever $p\forces\dot q\in Q$, then there is some $\dot q_*\in Q$ such that $p\forces\dot q=\dot q_*$. This $Q$ allows us to define the iteration, $\PP\ast\dot\QQ$ as $\{\tup{p,\dot q}\mid p\in\PP,\dot q\in Q\}$, and extend the mechanism by recursion as usual.\footnote{This had been done in great details in both \cite{IS} and \cite{sym}.} A finite support iteration of length $\delta$ is given by a sequence $\tup{\dot\QQ_\alpha\mid\alpha<\delta}$, and recursively defining $\PP_\alpha$ for $\alpha\leq\delta$ with $\PP_0=\{\1\}$ and the following conditions hold: $\dot\QQ_\alpha$ is a $\PP_\alpha$-name for a notion of forcing; $\PP_{\alpha+1}=\PP_\alpha\ast\dot\QQ_\alpha$; and for a limit ordinal $\alpha$, $p\in\PP_\alpha$ if and only if $p\restriction\beta\in\PP_\beta$ for all $\beta<\alpha$ and $\{\beta<\alpha\mid p(\beta)\neq\dot\1_{\QQ_\alpha}\}$ is finite.

Throughout this paper we will work in $\ZF$, unless stated otherwise. Given a set $A$, we use $|A|$ to denote its cardinal, which is either an initial ordinal when $A$ can be well-ordered, or else its \emph{Scott cardinal}.\footnote{The Scott cardinal of $A$ is the collection of least ranked sets which are in bijection with $A$.} We will write $|A|\leq|B|$ to mean that there is an injection from $A$ to $B$, and $|A|\leq^*|B|$ to mean that there is a partial surjection from $B$ onto $A$. The strict version, $|A|<|B|$ means ``$|A|\leq|B|$ and $|B|\nleq|A|$'', and $<^*$ is similarly defined. When discussing cardinals, Greek letters (mainly $\kappa$ and $\lambda$) will always denote well-ordered cardinals.

We write $\AC_X$ to mean ``every family of sets indexed by $X$ admits a choice function'',\footnote{We tacitly ignore the empty set in our families of sets in favour of readability. Formally, say that a family of sets admits a choice function if $\prod(\{A_i\mid i\in X\}\setminus\{\varnothing\})\neq\varnothing$.} and if $\kappa$ is a cardinal, then $\AC_{<\kappa}$ is the statement $\forall\lambda(\lambda<\kappa\to\AC_\lambda)$. We will write $\AC^X$ to mean that every family of sets whose individual sizes is $|X|$ admits a choice function. It is tempting to expect that we can replace the condition on the sizes by $\leq|X|$, but this is false: $\AC^2$ implies $\AC^4$, but $\AC^2$ does not imply $\AC^3$; therefore $\AC^4$ does not imply $\AC^3$ \cite{Jech:AC}. However, if $|X|=|X^2|$, and $\{A_i\mid i\in I\}$ is a family of sets such that $|A_i|\leq|X|$, we can replace $A_i$ by $X\times A_i$ and apply $\AC^X$. Throughout this paper, this will always be the case. We can combine these subscripts and superscripts in the obvious way, for example $\AC_\omega^\omega$ is the statement that ``every countable family of countable sets admits a choice function''.

The \emph{Principle of Dependent Choice for $\kappa$} is the statement that every $\kappa$-closed tree has a chain of order type $\kappa$ or a maximal element.\footnote{A tree is $\kappa$-closed if for every $\gamma<\kappa$, every chain of type $\gamma$ has an upper bound.} It is a standard fact that $\DC_\kappa$ implies both $\AC_\kappa$ and $\forall X(\kappa\leq|X|\lor|X|\leq\kappa)$. We write $\DC_{<\kappa}$ to mean that for every $\lambda<\kappa$, $\DC_\lambda$ holds.

We finish with the following transfer theorem from \cite{ccc}, which we will use as a black box for proving independence results.

Suppose that $V\subseteq W\subseteq V[G]$ are models of $\ZF$, with $G$ being a generic filter for some forcing notion in $V$, and $W$ a symmetric extension of $V$.\footnote{Symmetric extensions lie between the ground model and a generic extension. We will not appeal to those directly in this work, but \cite{sym} provides a complete treatment of the concept.} Given two structures $M\in V$ and $N\in W$, in some language, we say that $N$ is a \emph{symmetric copy} of $M$, if $V[G]\models M\cong N$.\footnote{We will generally do this in a way such that $W\models M\ncong N$.}

Let us fix a language $\cL$, a structure $M$, a group $\sG\subseteq\aut(M)$, and an ideal of subsets of $M$, $\cI$, containing all singletons. We say that $X\subseteq M^k$, for some $k<\omega$, is ``stable'' if for some $e\in\cI$, whenever $\pi\in\sG$ and $\pi\restriction e=\id$, then $\pi``X=X$.\footnote{In the case $k>1$, the action is pointwise on each coordinate.}

\begin{fact}[Transfer Theorem {\cite[Theorem~3.2]{ccc}}]\label{thm:transfer}
  For any a language, $\cL$, an $\cL$-structure, $M$, a group $\sG\subseteq\aut(M)$, and an ideal of subsets of $M$ extending the ideal of finite sets, $\cI$, there exists a symmetric extension, $W\subseteq V[G]$, in which there is a symmetric copy $N$ of $M$, such that for some $\iota\colon M\cong N$ in $V[G]$, if $A\subseteq N^k$ lies in $W$, then there is some $B\subseteq M^k$ which is stable and $\iota``B=A$.

Moreover, we may also assume $\DC_{<\kappa}$ holds in $W$ if $\cI$ is a $\kappa$-complete ideal, and that $\power(M)^V=\power(M)^W=\power(M)^{V[G]}$.
\end{fact}
\section{Finite support iterations of the countable chain condition}\label{sect:ccc}
\subsection{Background}
Karagila and Schweber defined several variants of ``ccc'' in \cite{ccc} and studied their consequences. They began with these three:
\begin{definition}[{\cite[Definition~4.1]{ccc}}]
  Let $\PP$ be a notion of forcing.
  \begin{description}
  \item[$\ccc_1(\PP)$] Every maximal antichain in $\PP$ is countable.
  \item[$\ccc_2(\PP)$] Every antichain in $\PP$ is countable.
  \item[$\ccc_3(\PP)$] Every predense subset of $\PP$ contains a countable predense subset.
  \end{description}
\end{definition}
We will often write ``ccc'' to indicate that we are working in $\ZFC$, where all three are equivalent, reserving $\ccc_i$ to suggest that we are working in $\ZF$. In \cite{ccc} it was shown that $\ccc_3(\PP)$ implies $\ccc_2(\PP)$, which in turn implies $\ccc_1(\PP)$, and that $\ZF+\DC$ is not strong enough to prove any of these implications can be reversed. However, under $\DC_{\omega_1}$, $\ccc_2(\PP)$ is equivalent to $\ccc_3(\PP)$.

It was also shown there that $\ccc_3$ partial orders tend to have the expected preservation properties: $\DC$ is preserved, no cardinals larger than $\omega_1$ can be collapsed, no cofinalities larger than $\omega_1$ are changed (and if $\omega_1$ is regular, then it is preserved as well).\footnote{It is still open whether it is possible to collapse $\omega_1$ with a $\ccc_3$ forcing when $\omega_1$ is singular.} On the other hand, it is consistent with $\ZF+\DC$ that a $\ccc_2$ forcing will collapse $\omega_1$ and violate $\DC$ in the process. All this is to say, $\ccc_3$ is a ``good definition'' for ccc in $\ZF$, and even more so in $\ZF+\DC$, as it is also equivalent to Mekler's definition in this case.\footnote{See \autoref{subs:mekler} for full details.}

Ikegami and Schlicht \cite{IS} proposed a different variant of ccc forcing which they call ``narrow'', which is equivalent to the statement ``all well-orderable antichains in the Boolean completion are countable''. They define a stricter notion of ``uniformly narrow'' and use that to define a ``uniform iteration'' which satisfies the wanted property: the finite support uniform iteration of uniformly narrow forcing notions is itself a uniformly narrow forcing.

However, as noted before, Schlicht had remarked that $\ccc_3$ does not have an iteration theorem. We formalise this claim and prove it in the next part. Next, we show that $\DC$ is sufficient for proving an iteration theorem for the class of $\ccc_3$ forcings. Namely, a finite support iteration of $\ccc_3$ forcings is $\ccc_3$ when working in $\ZF+\DC$.\footnote{This, perhaps, is an additional argument to the Asper\'o--Karagila thesis from \cite{properness} that studying the properties of the real numbers should be done in $\ZF+\DC$ as a minimal theory.}
\subsection{Iterability in \texorpdfstring{$\ZF$}{ZF}}\label{subs:no-iter}
\begin{theorem}\label{thm:cut}
  Suppose that the class of $\ccc_3$ forcings is closed under finite support iterations of length $\omega$. Then the countable union of countable sets is countable.
\end{theorem}
\begin{proof}
  Let $\{A_n\mid n<\omega\}$ be a countable family of countable sets. For each $n<\omega$ let $\QQ_n=\{\1\}\cup A_n\cup\{A_n\}$ with the order $q\leq_{\QQ_n} p$ if and only if $q=p$ or $p=\1$. In other words, the forcing simply picks an element out of $A_n$, or $A_n$ itself.

  It is easy to see that $\QQ_n$ is countable, so it is $\ccc_3$. Moreover, as the forcing is atomic, it does not add any subsets to the universe, so it preserves being $\ccc_3$.

  Let $\PP$ be the finite support iteration of these $\QQ_n$, and suppose that $\PP$ is $\ccc_3$. Let $D$ be the following set, \[\{p\in\PP\mid\exists n, p(n)=A_n\land\dom p\in\omega\}.\]
  Observe that $D$ is a dense set, as any condition can be extended to have domain that is an initial segment, and if needed, extended by a single point to have $p(n)=A_n$. Let $D'=\{p_k\mid k<\omega\}\subseteq D$ be a countable predense subset.

  We claim that for every $n<\omega$ and $a\in A_n$, there is some $p_k\in D'$ such that $p_k(n)=a$. To see this, fix $n$ and $a\in A_n$, and let $p\in\PP$ be a condition such that:
  \begin{enumerate}
  \item $\dom p=n+1$,
  \item $p(i)\in A_i$ for all $i<n$,
  \item $p(n)=a$.
  \end{enumerate}
  Since $D'$ is a predense set, there is some $k<\omega$ such that $p_k$ is compatible with $p$. As both $p$ and $p_k$ are defined on an initial segment of the iteration, it must be that either $p\subseteq p_k$ or $p_k\subseteq p$. However, as $p_k(i)=A_i$ for some $i$, it must be that $p\subseteq p_k$. Therefore, $p_k(n)=a$ as wanted.

  This defines an injection from $\bigcup_{n<\omega}A_n$ into $\omega\times\omega$ given by $a\mapsto\tup{k,n}$, the lexicographically minimal pair such that $p_k(n)=a$.\end{proof}
\begin{corollary}
  If the class of $\ccc_3$ forcings is closed under finite support iterations of length $\omega$, then $\AC_\omega^\omega$ holds.
\end{corollary}
\begin{proof}
  Let $\{A_n\mid n<\omega\}$ be a countable family of countable sets. By \autoref{thm:cut} we can enumerate $\bigcup_{n<\omega}A_n$ as $\{x_i\mid i<\omega\}$ and so define the choice function $F(A_n)=x_i$ where $i$ is the least such that $x_i\in A_n$.
\end{proof}
The proof of \autoref{thm:cut} seems as though it would generalise immediately to any length of iteration. That is, assuming that the class of $\ccc_3$ forcing notions is closed under any finite support iteration, then the union of a family of countable sets of size $\kappa\geq\omega$ has cardinality $\kappa$, and in particular any well-orderable family of countable sets admits a choice function.

However, if $\{A_\alpha\mid\alpha<\kappa\}$ is a family of countable sets, repeating the above proof can only guarantee that our predense set meets countably many of them, which will be compatible with choosing from any other coordinate. Indeed, as will be shown in \autoref{thm:dc-it}, $\DC$ is already strong enough to prove that the finite support iterations of $\ccc_3$ forcings is $\ccc_3$, and since $\lnot\AC^\omega_{\omega_1}$ is consistent with $\DC$, this shows that the above proof is close to optimal.

\subsection{Properness and Mekler's definition}\label{subs:mekler}
Assuming the Axiom of Choice, we can talk about ccc forcings through the lens of properness. The definition of properness usually talks about countable $M\prec H(\theta)$ for all (or some) sufficiently large regular $\theta$. Properness was briefly studied in \cite{properness}, where a definition of $H(\theta)$ was suggested as $\{x\mid\theta\nleq^*\tcl(x)\}$, this definition is problematic and the claim that this is a model of $\ZF^-$ is false.\footnote{For example, it is consistent that for two sets, $A$ and $B$, neither one maps onto $\omega_1$, but $A\times B$ can be mapped onto $\omega_1$.} Other possible definitions were given by Rathjen and Lubarsky \cite{RL} and by Goldberg \cite{Goldberg:even}. Both of these definitions were analysed by Jeon \cite{reinhardt-powerset}. Jeon points out that while Goldberg's definition will always satisfy Zermelo's set theory without Power Set (and Choice), it is not clear if it will always satisfy Collection or even Replacement.

So why do we use $H(\theta)$? The main appeal, in the context of $\ZFC$ at least, is that it allows us to not worry about our use of Replacement. Jeon's result is a stark reminder that in the choiceless context we must be very diligent about what we need and what we have available in our universe.

Nevertheless, the work done in \cite{properness} is not invalidated by this issue. The motivation behind the definition of properness using countable elementary substructures is that the transitive collapse of our countable model commutes with its generic extension. This does not require us to assume Replacement in our structure and can be easily remedied by taking elementary submodels of suitable $V_\alpha$. This is formalised in the definition of a \emph{sufficient ordinal} given below.

\begin{definition}
  For a forcing notion $\PP$ we say that a limit ordinal, $\alpha$, is a \emph{sufficient ordinal (for $\PP$)} if $\PP\in V_\alpha$, $V_\alpha$ reflects the forcing theorem for $\PP$,\footnote{Recall that the Reflection Principle is a theorem of $\ZF$.} and for every $V$-generic filter, $G\subseteq\PP$, $V_\alpha[G]=V[G]_\alpha$. If $M\prec V_\alpha$ is a countable elementary submodel with $\PP\in M$, we say that $M$ is a \emph{sufficient ($\alpha$-)model (for $\PP$)}.
\end{definition}

\begin{lemma}\label{lemma:sufficient-ord}
  If $\alpha$ is sufficient for $\PP\ast\dot\QQ$ and $G\subseteq\PP$ is $V$-generic, then $\alpha$ is sufficient for $\PP$ and $V[G]\models``\alpha$ is sufficient for $\dot\QQ^G$''.
\end{lemma}
\begin{proof}
  Let us fix $G\ast H$, a $V$-generic filter for $\PP\ast\dot\QQ$, and let $\alpha$ be a sufficient ordinal for $\PP\ast\dot\QQ$. It will be easier for us to set the following conventions:
  \begin{enumerate}
  \item For a $\PP$-name, $\dot x$, we recursively define $\dot x_\QQ$ to be the $\PP\ast\dot\QQ$-name,
    \[\dot x_\QQ=\{\tup{\tup{p,\1_\QQ},\dot y_\QQ}\mid\tup{p,\dot y}\in\dot x\}.\]
  \item For a $\PP\ast\dot\QQ$-name, $\dot x$, we recursively define $[\dot x]$ to be the $\PP$-name (for a $\dot\QQ$-name),
    \[[\dot x] = \{\tup{p,\tup{\dot q,[\dot y]}^\bullet}\mid\tup{\tup{p,\dot q},\dot y}\in\dot x\}.\]
  \item For a $\PP\ast\dot\QQ$-name, $\dot x$, and a condition $\tup{p,\dot q}\in\PP\ast\dot\QQ$, we recursively define $\dot x_{p,\dot q}$ to be the $\PP$-name,
    \[\dot x_{p,\dot q}=\{\tup{\bar p,\dot y_{p,\dot q}}\mid\bar p\leq_\PP p,\exists\bar q,\tup{\bar p,\bar q}\compt\tup{p,\dot q},\tup{\bar p,\bar q}\forces_{\PP\ast\dot\QQ}\dot y\in\dot x \text{ and }\dot y\text{ appears in }\dot x\}.\]
  \end{enumerate}
  Note that $[\dot x_\QQ]$ is the canonical $\QQ$-name in $V[G]$ for $\dot x^G$ and that a standard induction shows that $\dot x_{p,\dot q}\in V_\alpha$ whenever $\dot x\in V_\alpha$ and $\tup{p,\dot q}\in\PP\ast\dot\QQ$. To give a semblance of intuition for the definition of $\dot x_{p,\dot q}$, whenever $\dot x$ is a $\PP\ast\dot\QQ$-name for a subset of $V$, then $\tup{p,\dot q}\forces_{\PP\ast\QQ}\dot x\subseteq(\dot x_{p,\dot q})_\QQ$. Our goal is to show that it is in fact a $\PP$-name for $\dot x^{G\ast H}$ whenever $\dot x\in V_\alpha$.

  Let us now show that $\alpha$ is sufficient for $\PP$. Easily, $\PP\in V_\alpha$. First, we have the following inclusions,
  \[V_\alpha\subseteq V[G]_\alpha\subseteq V[G\ast H]_\alpha = V_\alpha[G\ast H]=V_\alpha[G][H] \supseteq V_\alpha[G].\]

  Let us show that $V[G]_\alpha\subseteq V_\alpha[G]$. Let $x\in V[G]_\alpha$, we will show that $x$ has a $\PP$-name in $V_\alpha$. As mentioned above, $V[G]_\alpha\subseteq V_\alpha[G\ast H]$, and so let $\dot x\in V_\alpha$ be a $\PP\ast\dot\QQ$-name for $x$.

  As $x\in V[G]$, there is a condition $\tup{p,\dot q}\in G\ast H$ such that $\tup{p,\dot q}\forces_{\PP\ast\dot\QQ}\dot x=\dot x'_\QQ$ for some $\PP$-name $\dot x'$. We claim that $\dot x_{p,\dot q}^G=x$, which will be the wanted $\PP$-name in $V_\alpha$. As the interpretation of $(\dot x'_\QQ)^{G\ast H}$ does not actually depend on $H$ itself, it must be that the membership of $\dot x$ does not depend on any condition stronger than $\dot q$ in the second coordinate. Therefore, $\dot x_{p,\dot q}^G=\dot x^{G\ast H}=x$ as wanted.

  In the other direction, it is a general fact of forcing that if $\dot x$ has a name rank ${<}\alpha$, then $\dot x^G\in V[G]_\alpha$. Moreover, if $\dot x\in V_\alpha$, then its name rank is ${<}\alpha$ as well. That is to say, $V_\alpha[G]\subseteq V[G]_\alpha$. Therefore, $V[G]_\alpha=V_\alpha[G]$ as wanted.

  To show that $\alpha$ is sufficient for $\PP$ it remains to show that $V_\alpha$ reflects the forcing theorem for $\PP$. Namely, $V_\alpha[G]\models\varphi$ if and only if there is some $p\in G$ such that $p\forces_\PP\varphi$. The proof is the usual proof of the forcing theorem, noting that since $\alpha$ is a limit ordinal and $\PP\in V_\alpha$, $V_\alpha$-genericity is the same as $V$-genericity, as well as that the equality of $V[G]_\alpha=V_\alpha[G]$ guarantees the existence of names necessary to instantiate any quantifiers.

  Finally, let us check that $V[G]\models``\alpha$ is sufficient for $\dot\QQ^G$''. Easily, $\dot\QQ\in V_\alpha$, so $\dot\QQ^G\in V[G]_\alpha$. Next, by the fact that $\alpha$ is sufficient for $\PP$ and $\PP\ast\dot\QQ$, \[V[G]_\alpha[H]=V_\alpha[G][H]=V_\alpha[G\ast H]=V[G\ast H]_\alpha=V[G][H]_\alpha.\]
  To see that $V_\alpha[G]$ reflects the forcing theorem for $\dot\QQ^G$, we have that
  \begin{align*}
    V_\alpha[G][H]\models\varphi(([\dot x]^G)^H) &\iff V_\alpha[G\ast H]\models\varphi(\dot x^{G\ast H})\\
                                                 &\iff \exists\tup{p,\dot q}\in G\ast H, \tup{p,\dot q}\forces_{\PP\ast\dot\QQ}\varphi(\dot x)\\
                                                 &\iff \exists\tup{p,\dot q}\in G\ast H, p\forces_\PP``\dot q\forces_\QQ\varphi([\dot x])"\\
                                                 &\iff V_\alpha[G]\models \dot q^G\forces_\QQ\varphi([\dot x]^G).\qedhere
  \end{align*}
\end{proof}

\begin{definition}
  Let $\PP$ be a notion of forcing and $M$ a sufficient model. We say that $p\in\PP$ is \emph{$M$-generic} if whenever $D\in M$ is a predense set, then $D\cap M$ is predense below $p$. Namely, for every $q\leq p$ there is some $r\in D\cap M$ which is compatible with $q$.\footnote{Of course, we may replace the condition that $D$ is predense by either dense or dense open, and if the Axiom of Choice is assumed, maximal antichain will work too.}
\end{definition}
Easily, any extension of an $M$-generic condition is itself $M$-generic, and if the set of $M$-generic conditions is predense below $p$, then $p$ is also an $M$-generic condition.
\begin{lemma}\label{lemma:proper-elem}
  Let $M$ be a sufficient $\alpha$-model for a forcing notion $\PP$. If $G\subseteq\PP$ is a $V$-generic filter containing an $M$-generic condition, then $M[G]\prec V[G]_\alpha=V_\alpha[G]$.
\end{lemma}
\begin{proof}
  If $\dot x^G\in M[G]$, we may assume $\dot x\in M$. As $\alpha$ is sufficient, $V_\alpha$ satisfies the forcing theorem. Therefore,
  \begin{align*}
    M[G]\models\varphi(\dot x^G) &\iff \exists p\in G\cap M, M\models p\forces\varphi(\dot x)\\
                                 &\iff \exists p\in G\cap M, V_\alpha\models p\forces\varphi(\dot x)\\
                                 &\iff V_\alpha[G]\models\varphi(\dot x^G).\qedhere
  \end{align*}
\end{proof}

The following is an immediate consequence of the previous lemmas.
\begin{corollary}\label{cor:proper-sufficients}
  Suppose $M$ is an $\alpha$-sufficient model for $\PP\ast\dot\QQ$ and $G\subseteq\PP$ is a $V$-generic filter containing an $M$-generic condition. Then $M$ is sufficient for $\PP$ and $M[G]$ is sufficient for $\dot\QQ^G$ in $V[G]$.
\end{corollary}
\begin{proof}
  By \autoref{lemma:sufficient-ord}, $\alpha$ is sufficient for $\PP$ and for $\dot\QQ^G$ in $V[G]$, and therefore $M$ is sufficient for $\PP$. By \autoref{lemma:proper-elem}, $M[G]\prec V_\alpha[G]$. Since $\PP$, $\PP\ast\dot\QQ\in M$, there is some $\dot Q\in M$ such that $\PP\ast\dot\QQ\cong\PP\ast\dot Q$. Therefore $\dot\QQ^G\in M[G]$. So $M[G]$ is sufficient for $\dot\QQ^G$ in $V[G]$ as wanted.
\end{proof}

As there is a closed and unbounded class of sufficient ordinals for any given $\PP$, we can use this notion to rewrite the definition of properness.
\begin{definition}[Properness]
$\PP$ is proper if for any large enough sufficient ordinal, $\alpha$, if $M$ is a sufficient $\alpha$-model, then every $p\in\PP\cap M$ has an $M$-generic extension.
\end{definition}

Of course, this definition requires $\DC$ to hold to make sense, as noted in \cite{properness}, otherwise there are no sufficient models, and every forcing is proper for vacuous reasons.

We say that $\PP$ satisfies \emph{Mekler's definition of ccc} when for every sufficient model, $M$, $\1_\PP$ is an $M$-generic condition \cite{Mekler}. As every filter contains $\1_\PP$, it means that in the case of Mekler's ccc, \autoref{lemma:sufficient-ord} states that $M[G]\prec V_\alpha[G]$ whenever $M$ is an $\alpha$-sufficient model and $G$ is a $V$-generic filter.

\begin{fact}[$\DC$, {\cite[Proposition~5.1]{ccc}}]
  $\PP$ is a $\ccc_3$ forcing if and only if it satisfies Mekler's definition of ccc.
\end{fact}

\subsection{Iterability in \texorpdfstring{$\ZF+\DC$}{ZF+DC}}

The rest of the section will be dedicated to the proof of the following theorem.

\begin{theorem}[$\DC$]\label{thm:dc-it}
  Finite support iterations of $\ccc_3$ forcings are $\ccc_3$.
\end{theorem}
To prove the theorem we break down the two cases: two-step iterations for the successor case; and limit steps. Throughout the rest of the section we work in $\ZF+\DC$.

\begin{lemma}\label{lemma:succ}
  Suppose that $\PP$ is proper and $\1\forces_\PP``\dot\QQ$ is proper''. Let $M$ be sufficient for $\PP\ast\dot\QQ$, and $p\in\PP$ is $M$-generic such that $p\forces_\PP``\dot q$ is $\check M[\dot G]$-generic''. Then $\tup{p,\dot q}$ is $M$-generic for $\PP\ast\dot\QQ$.
\end{lemma}
\begin{proof}
  Let $\PP$, $\dot\QQ$, $p$, $\dot q$, and $M$ be as stated. Let $D\in M$ be a dense open subset of $\PP\ast\dot\QQ$. Let us verify that $D\cap M$ is predense below $\tup{p,\dot q}$. Note that $\dom D\in M$ is a dense open subset of $\PP$, and that $D$ is a $\PP$-name for a dense subset of $\dot\QQ$.

  Let $\tup{p_1,\dot q_1}\leq\tup{p,\dot q}$. As $p_1\leq_\PP p$, it must agree that $\dot q$ is $\check M[\dot G]$-generic. Since $D$ is a name for a dense subset of $\dot\QQ$, $p_1\forces_\PP\exists\dot q_0\in D\cap\check M[\dot G],\dot q_0\compt\dot q_1$. By extending $p_1$ if necessary, we may assume that for some $\dot q_0\in M$, $p_1\forces_\PP\dot q_0\in D\land\dot q_0\compt\dot q_1$.

  Consider the set $F=\{\bar p\in\PP\mid\exists\dot q_0',\tup{\bar p,\dot q_0'}\in D,\bar p\nforces_\PP\dot q'_0\neq\dot q_0\}$. We claim that $F$ is predense below $p_1$. If $p'\leq_\PP p_1$, then $p'\forces_\PP\dot q_0\in D$, in particular, there is some $r\leq_\PP p'$ and some $\tup{\bar p,\dot q'_0}\in D$ such that $r\leq_\PP\bar p$ and $r\forces_\PP\dot q'_0=\dot q_0$. In particular, such $\bar p\in F$ and it is compatible with $p'$.

  As $F$ is definable from $\dot q_0$, $D$, and $\PP$, we have that $F\in M$. Therefore, by the $M$-genericity of $p_1$, there is some $p_0\in F\cap M$ such that $p_1\compt p_0$. As $p_0\in F\cap M$, by elementarity of $M$ it must be that there is some $\dot q'_0\in M$ witnessing this. In particular, $\tup{p_0,\dot q'_0}\in D\cap M$, and $\tup{p_1,\dot q_1}\compt\tup{p_0,\dot q'_0}$. Therefore, $D\cap M$ is predense below $\tup{p,\dot q}$ as wanted.
\end{proof}
An immediate corollary from the above provides us with the two-step iteration.
\begin{corollary}\label{cor:ccc*ccc}
  If $\PP$ is proper and $\1\forces_\PP``\dot\QQ$ is proper'', then $\PP\ast\dot\QQ$ is proper. Similarly, if $\PP$ is $\ccc_3$ and $\1\forces_\PP``\dot\QQ$ is $\ccc_3$'', then $\PP\ast\dot\QQ$ is $\ccc_3$.\qed
\end{corollary}
\begin{lemma}\label{lemma:ctble-cf}
  Let $\delta$ be a limit ordinal and let $\PP_\delta=\bigast_{\alpha<\delta}\dot\QQ_\alpha$ be the finite support iteration of $\ccc_3$ forcings. If for all $\gamma<\delta$, $\PP_\gamma$ is $\ccc_3$, then $\PP_\delta$ is $\ccc_3$.
\end{lemma}
\begin{proof}
  Let $D$ be a predense subset of $\PP_\delta$, and let $M$ be a sufficient model such that $D\in M$. For each $\alpha\in M\cap\delta$, the set $D\restriction\alpha = \{p\restriction\alpha\mid p\in D\}$ is predense in $\PP_\alpha$. Otherwise, there is some $p_\alpha\in\PP_\alpha$ which is incompatible with all $p\restriction\alpha$ where $p\in D$. However, $p_\alpha^\frown\1_{[\alpha,\delta)}$ is compatible with some $p\in D$, which means that $p_\alpha$ is compatible with $p\restriction\alpha$.

  Since $\PP_\alpha$ is $\ccc_3$, $D\restriction\alpha$ has a countable predense subset, $E_\alpha = \{p^\alpha_n\mid n<\omega\}$. Since $D$ and $\alpha$ are in $M$, by elementarity we may assume $E_\alpha\in M$. In turn, this implies that $E_\alpha\subseteq M$.\footnote{Recall that if $M\prec V_\alpha$ is countable, and $E\in M$ is a countable set, then $E\subseteq M$.}

  For each  $\alpha\in M\cap\delta$ and $n<\omega$, if $p^\alpha_n\in E_\alpha$, then there is some $q\in D$ such that $p=q^\alpha_n\restriction\alpha$. As $p,D$, and $\alpha$ are all in $M$, we can find such $q\in M$ by elementarity, and so we may fix one as $q^\alpha_n$. Therefore, we can define the sets $F_\alpha=\{q^\alpha_n\mid n<\omega\}$ for $\alpha\in M\cap\delta$, and $F=\bigcup_{\alpha\in M\cap\delta}F_\alpha$ is a countable subset of $D$. If we show that $F$ is predense in $\PP_\delta$, the proof will be complete.

  Let $p\in\PP_\delta$ be any condition and let $\alpha=\max\supp(p)+1$, which exists as $p$ has a finite support. There are two cases to deal with.
\begin{description}
\item[$\boldsymbol{\alpha<\sup M\cap\delta}$] Let $\beta\in M\cap\delta$ such that $\alpha<\beta$. Then there is some $p_n^\beta\in E_\beta$ which is compatible (in $\PP_\beta$) with $p\restriction\beta$. Therefore, $q_n^\alpha\in F_\alpha\subseteq F$ is compatible with $p=p\restriction\beta^\frown\1_{[\beta,\delta)}$.

\item[$\boldsymbol{\alpha\geq\sup M\cap\delta}$] Let $\alpha'=\max(\supp(p)\cap M)+1$, then $p\restriction\alpha'^\frown\1_{[\alpha',\delta)}$ is compatible, by the previous case, with some $q\in F$. However, $\supp(q)\subseteq M$ and therefore $q$ must be compatible with $p$ as well.
\end{description}
  In either case, $p$ is compatible with some element of $F$, and therefore it is a predense set as wanted.
\end{proof}

\begin{proof}[Proof of \autoref{thm:dc-it}]
  We prove this by induction on the length of the iteration. For $\alpha=\beta+1$ we have that $\PP_{\beta+1}=\PP_\beta\ast\dot\QQ_\beta$, so by \autoref{cor:ccc*ccc} the iteration is $\ccc_3$, and if $\alpha$ is a limit ordinal, then by \autoref{lemma:ctble-cf} the iteration is $\ccc_3$.
\end{proof}

\section{Knaster does not imply \texorpdfstring{$\ccc_3$}{ccc3}}\label{sect:K}
The countable chain condition can be stated as ``given uncountably many conditions, two of them must be compatible''. This can be strengthened to the \emph{Knaster property} which states that given uncountably many conditions, uncountably many of them are pairwise compatible. While it is consistent with $\ZFC$ that ccc and Knaster are equivalent,\footnote{E.g., Martin's Axiom proves that the two notions agree.} it is also consistent that there are ccc forcings which are not Knaster.\footnote{E.g., a Suslin tree is ccc, but it is not Knaster.}

Much like with the chain condition, one can ask about ``the real meaning of Knaster'' in $\ZF$. We first explore some versions of ``uncountable'', and in \autoref{thm:knaster} we show that a Knaster forcing need not be $\ccc_3$, even if $\ZF+\DC$ is assumed.

\subsection{Uncountability revised}
The standard definition of the Knaster property says that given uncountably many conditions, there are uncountably many of them which are pairwise compatible. This is a strengthening of the countable chain condition. In the choiceless context, however, the meaning of uncountably can be changed. Taken literally, it just means ``not countable'', but other interpretations can be made. We will focus on these common ones.
\begin{definition} Given a set $X$ we say that it is $(a)$-uncountable, for $a\in\{1,2,3,4\}$, if it satisfies the corresponding statement:
  \begin{enumerate}
  \item $\aleph_1\leq|A|$.
  \item $\aleph_0<|A|$.
  \item $\aleph_0<^*|A|$.
  \item $|A|\nleq\aleph_0$.
  \end{enumerate}
\end{definition}
\begin{proposition}
  For $i\leq j$, every $(i)$-uncountable set is $(j)$-uncountable as well.
\end{proposition}
\begin{proof}
  Immediate.
\end{proof}
\begin{proposition}
  None of the implications can be reversed in $\ZF$.\footnote{These are all well-known facts. We provide a sketch of the proofs here for the completeness of the paper.}
\end{proposition}
\begin{proof}
  To see that $(4)$-uncountable does not imply $(3)$-uncountable, we use \autoref{thm:transfer} with the structure $\omega$, $\sG$ the full symmetry group of $\omega$, and $\cI=[\omega]^{<\omega}$. As the only stable subsets of $\omega$ are the finite and co-finite ones, the symmetric copy, $A$, is amorphous. Namely, all of its subsets are finite or co-finite. Easily, such $A$ is $(4)$-uncountable. To see that it is not $(3)$-uncountable, note that if for any set $X$, $f\colon X\to\omega$ is a surjection, then the preimage of the even integers is an infinite co-infinite subset of $X$. As no such set exists for $A$, it is not $(3)$-uncountable.

  Next, to show that $(3)$-uncountable does not imply $(2)$-uncountability, note that in Cohen's first model there exists an infinite Dedekind-finite set\footnote{Recall that a set $X$ is \emph{Dedekind-finite} if its cardinal is incomparable with $\aleph_0$.} (which is $(4)$-uncountable) of real numbers. As every infinite set of real numbers can be mapped onto $\omega$, it must be that this set is $(3)$-uncountable, but as it is Dedekind-finite, it cannot be $(2)$-uncountable.

  Finally, using \autoref{thm:transfer}, let $M$ be the structure $\omega_1$ in the empty language, letting $\sG$ be its symmetry group, and letting $\cI=[\omega_1]^{<\omega_1}$. Then there is a symmetric copy of $M$, $N$, whose subsets are exactly those stable under fixing countably many points. It is not hard to check that there is no subset of $N$ which is uncountable and co-uncountable. In particular, it must be the case that $|N|$ is incomparable with $\aleph_1$, and so it is not $(1)$-uncountable, and since $N$ is both infinite and has a countably infinite subset, it is $(2)$-uncountable.
\end{proof}
The careful reader will see that by \autoref{thm:transfer} the final model satisfies $\DC$, as the ideal $[\omega_1]^{<\omega_1}$ is $\sigma$-complete. Therefore, $\ZF+\DC$ does not prove that $(2)$-uncountable sets are $(1)$-uncountable sets. More is true, as stated next.
\begin{proposition}
If $\DC$ holds, then $(2)$-, $(3)$-, and $(4)$-uncountable are all equivalent. If $\DC_{\omega_1}$ holds, then all four notions are equivalent.
\end{proposition}
\begin{proof}
  Assuming $\DC$, it is easy to check that if $X$ is an infinite set, then $\aleph_0\leq|X|$. Therefore, every infinite set is $(4)$-uncountable if and only if it is $(2)$-uncountable. Similarly, if $\DC_{\omega_1}$ holds, then $\omega_1$ injects into any uncountable set.
\end{proof}
\subsection{Choiceless Knaster properties}
\begin{definition}
  Let $\PP$ be a notion of forcing. We say that $\PP$ has the $\K(i,j)$ property if whenever $A\subseteq\PP$ is $(i)$-uncountable, it has a $(j)$-uncountable subset of pairwise compatible conditions.
\end{definition}
It is easy to see that when $i$ grows large, more sets are considered, and the property is stronger, and as $j$ is smaller it is harder to find the witness, and so we are again strengthening the property. Taken at face value, the definition of Knaster would be $\K(4,4)$. But, if we assume $\DC$, this reduces to $\K(2,2)$.
\begin{theorem}\label{thm:knaster}
  It is consistent with $\ZF+\DC$ that there is a notion of forcing which is $\K(2,1)$ but not $\ccc_3$.
\end{theorem}
\begin{proof}
  We assume $\ZFC$ holds in the ground model. Let $\QQ^*$ be an uncountable notion of forcing which is Knaster, homogeneous, and for every $q\in\QQ^*\setminus\{\1_*\}$, the orbit of $q$ under $\aut(\QQ^*)$ is uncountable. For example, $\Add(\omega,\omega_1)$ is such forcing. Let $\QQ$ denote $\QQ^*\setminus\{\1_*\}$.

  We let $\PP$ be the lottery sum $\bigoplus_{\alpha<\omega_1}\QQ$. We let $\sG$ be the full support product $\prod_{\alpha<\omega_1}\aut(\QQ)$, and if $\pi=\tup{\pi_\alpha\mid\alpha<\omega_1}\in\sG$, then $\pi(\tup{\alpha,q}) = \tup{\alpha,\pi_\alpha q}$.

  Finally, let $\cI$ be the ideal generated by countable unions of summands. Namely, for $A\subseteq\PP$, $A\in\cI$ if and only if there is some $e\in[\omega_1]^{<\omega_1}$ such that $A\subseteq\bigoplus_{\alpha\in e}\QQ$. We let $\fix(e)=\{\pi\in\sG\mid\pi\restriction\bigoplus_{\alpha\in e}\QQ = \id\}$.

  Let $W$ be the symmetric extension satisfying $\ZF+\DC$ guaranteed by \autoref{thm:transfer} and let $\PP_*\in W$ be the symmetric copy of $\PP$. Moreover, \autoref{thm:transfer} allows us to assume that $\power(\PP)^V=\power(\PP)^W$. We shall denote by $A_*\subseteq\PP_*$ the subset corresponding to $A\subseteq\PP$ in $V$.

  Let $A_*\subseteq\PP_*$ be a $(2)$-uncountable set. In $V$, $A$ is stable under $\fix(e)$ for some countable $e$. So, either $A\subseteq\bigoplus_{\alpha\in e}\QQ$, in which case there is some $\alpha$ such that $A$ meets the $\alpha$th summand on an uncountable set; or else there is some $\alpha\notin e$ such that $A$ meets the $\alpha$th summand, in which case $A$ contains the uncountable orbit of a condition. In either case, $A$ contains an uncountable subset of a single summand. As $\QQ$ is Knaster, there is an uncountable $B\subseteq A$, which is contained in a single summand with index $\beta$, and every two conditions in $B$ are compatible. Therefore, $B_*$ is a $(2)$-uncountable subset of $A_*$. However, as $B$ is fixed pointwise by $\fix(\{\beta\})$, $B_*$ can be well-ordered, and therefore it is $(1)$-uncountable as wanted.

  To see that $\ccc_3$ fails, $\PP\setminus\{\1_\PP\}$ is certainly predense and stable under all automorphisms, however it does not have any countable predense set. Therefore, its symmetric copy does not have such countable predense set either.
\end{proof}
As $\DC_{\omega_1}$ implies that $\ccc_2$ forcings are $\ccc_3$, it follows that in the presence of $\DC_{\omega_1}$, Knaster does imply $\ccc_3$.
\section{Open questions}
This research was motivated by trying to understand Martin's Axiom in $\ZF$. Each definition of ccc gives rise to a different version of Martin's Axiom. But in order to force Martin's Axiom, one would naively want to try and mimic the $\ZFC$ proof, in which an iteration is defined. This motivated us to try and make sense of Schlicht's remark, as well as to understand the connection between the Knaster property and ccc in $\ZF$.

As Martin's Axiom has not been properly investigated in $\ZF$, there are many interesting questions that one can try and solve. We present a few that we think are particularly important for the general understanding of forcing axioms in $\ZF$.

\begin{question}Are any of the following a consequence of Martin's Axiom for $\ccc_3$ forcings or Martin's Axiom for $\ccc_2$ forcings?
  \begin{enumerate}
  \item $\DC$? $\DC_{\omega_1}$?
  \item $\ccc_2$ is equivalent to $\ccc_3$?
  \item $\K(i,j)$ implies $\ccc_3$ for some $i,j$?
  \item $\ccc_3$ implies $\K(i,j)$ for some $i,j$?
  \end{enumerate}
\end{question}
\begin{question}
  Assuming $\ZF+\DC$, is there a $\ccc_3$ forcing (e.g., a finite support iteration) which forces Martin's Axiom for $\ccc_3$ without necessarily forcing the Axiom of Choice?
\end{question}
\begin{question}
  Is the Axiom of Choice needed at all to prove that the two-step iteration of $\ccc_3$ forcings is $\ccc_3$?
\end{question}
\begin{question}
  What is the exact fragment of the Axiom of Choice that is equivalent to the iterability of $\ccc_3$ forcings?
\end{question}

Ikegami and Schlicht defined the notion of a ``uniformly narrow'' forcing, which essentially provides a function that enumerates all the well-orderable antichains in the Boolean completion. Using a similar concept of a ``uniform iteration'' they show that the uniform iteration of uniformly narrow forcing is again uniformly narrow.

\begin{question}
  Can we develop the theory of ``uniformly $\ccc_3$'' forcings (e.g., by having a function which provides enumerated countable predense subsets) and their iteration without the Axiom of Choice?
\end{question}

We finish with a question about the theory of proper forcing in $\ZF+\DC$ which is somewhat tangential to the main focus of this paper. Say that $p$ is a \emph{weakly $M$-generic condition} if  whenever $\dot x\in M$ is such that $p\forces\dot x\in\check V$, then $p\forces\dot x\in\check M$. Assuming $\ZFC$, weakly generic conditions are generic, and indeed we can weaken them even further to require that $\dot x$ is a name for an ordinal. This characterisation provides an easy and straightforward proof for \autoref{lemma:succ}. However, this lemma relies heavily on maximal antichains and well-orders.

\begin{question}
  Does $\ZF+\DC$ prove that a weakly $M$-generic condition is $M$-generic? Can we restrict this even further to names for ordinals?
\end{question}
\subsection*{Acknowledgements} The authors would like to thank Vincenzo Mantova for his suggestions on the introduction of this paper and Andreas Lietz for pointing out an implicit assumption in one of the lemmas. We would also like to thank the anonymous referee for their careful reading and helpful remarks.
\newpage
\providecommand{\bysame}{\leavevmode\hbox to3em{\hrulefill}\thinspace}
\providecommand{\MR}{\relax\ifhmode\unskip\space\fi MR }
\providecommand{\MRhref}[2]{%
  \href{http://www.ams.org/mathscinet-getitem?mr=#1}{#2}
}
\providecommand{\href}[2]{#2}

\end{document}